\documentclass[12pt]{amsart}

\usepackage{amssymb}
\usepackage{amsthm}
\usepackage{mathtools}
\usepackage{tikz-cd}
\usepackage{color}
\usepackage{hyperref}
\usepackage{upref}
\usepackage{verbatim}
\RequirePackage{mathrsfs}
\RequirePackage{paralist}

\allowdisplaybreaks

\newtheorem{thm}{Theorem}[section]

\newtheorem{cor}[thm]{Corollary}
\newtheorem{lem}[thm]{Lemma}

\theoremstyle{definition}
\newtheorem{defn}[thm]{Definition}

\numberwithin{equation}{section}

\newcommand{\secref}[1]{Section~\textup{\ref{#1}}}

\newcommand{\thmref}[1]{Theorem~\textup{\ref{#1}}}

\newcommand{\lemref}[1]{Lemma~\textup{\ref{#1}}}

\newcommand{\diagref}[1]{Diagram~\textup{\ref{#1}}}

\newcommand\I{\mathscr{I}}

\newcommand\K{\mathcal{K}}

\renewcommand\H{\mathcal{H}}
\newcommand{\Inn}{\mathscr{I}^{N}}

\renewcommand{\epsilon}{\varepsilon}

\newcommand{\inv}{^{-1}}

\newcommand{\what}{\widehat}

\newcommand{\id}{\text{id}}

\renewcommand{\subset}{\subseteq}

\newcommand{\xt}{\otimes}

\newcommand{\oalpha}{\overset{\sim}{\alpha}}
\newcommand{\Ext}{\text{Ext}}
\newcommand{\Ind}{\text{Ind}}
\newcommand{\Res}{\text{Res}}

    \newcommand{\iso}{isomorphism}
    \newcommand{\sur}{surjection}

\newcommand{\pa}{\ensuremath{\Phi_A}}
\newcommand{\pai}{\ensuremath{\Phi_{A/I}}}
\newcommand{\lag}{\ensuremath{\ensuremath{\Lambda_{A\rtimes G}}}}
\newcommand{\laig}{\ensuremath{\ensuremath{\Lambda_{A/I\rtimes G}}}}
\newcommand{\iag}{\ensuremath{\ensuremath{i_{A\rtimes G}}}}
\newcommand{\iaig}{\ensuremath{\ensuremath{i_{A/I\rtimes G}}}}
\newcommand{\qgg}{\ensuremath{\ensuremath{q\rtimes G\rtimes G}}}
\newcommand{\qgrg}{\ensuremath{\ensuremath{q\rtimes G\rtimes_r G}}}
\newcommand{\ja}{\ensuremath{\ensuremath{j_A}}}
\newcommand{\ia}{\ensuremath{\ensuremath{j_I}}}
\newcommand{\jai}{\ensuremath{\ensuremath{j_{A/I}}}}
\newcommand{\jg}{\ensuremath{\ensuremath{j_G}}}
\newcommand{\agg}{\ensuremath{\ensuremath{A\rtimes G\rtimes G}}}
\newcommand{\agrg}{\ensuremath{\ensuremath{A\rtimes G\rtimes_r G}}}
\newcommand{\igg}{\ensuremath{\ensuremath{I\rtimes G\rtimes G}}}
\newcommand{\igrg}{\ensuremath{\ensuremath{I\rtimes G\rtimes_r G}}}
\newcommand{\aigg}{\ensuremath{\ensuremath{A/I\rtimes G\rtimes G}}}
\newcommand{\aigrg}{\ensuremath{\ensuremath{A/I\rtimes G\rtimes_r G}}}

\renewcommand{\l}{\left}
\renewcommand{\r}{\right}
\newcommand{\Lad}{\text{Lad}}
\renewcommand{\o}{\overline}

\renewcommand{\)}{\textup)}

\newcommand{\cst}{\ensuremath{C^*}}

\newcommand{\spn}{\operatorname{span}}

\begin{document}
\begin{abstract}
  Given a coaction $\delta$ of a locally compact group $G$ on a $\cst$-algebra $A$, we study the relationship between two different forms of coaction invariance of ideals of $A$ and the ideals of the corresponding crossed product $\cst$-algebra $A \rtimes_{\delta} G$.  In particular, we characterize when these two notions of invariance are equivalent.  
\end{abstract}

\title[Two Notions of Coaction Invariance of Ideals]{Comparing Two Notions of Coaction Invariance of Ideals in C*-Algebras}

\author[Gillespie]{Matthew Gillespie}
\address{School of Mathematical and Statistical Sciences
\\Arizona State University
\\Tempe, Arizona 85287}
\email{mjgille1@asu.edu}

\author[Jones]{Benjamin Jones}
\address{School of Mathematical and Statistical Sciences
\\Arizona State University
\\Tempe, Arizona 85287}
\email{brjone16@asu.edu}

\author[Kaliszewski]{S. Kaliszewski}
\address{School of Mathematical and Statistical Sciences
\\Arizona State University
\\Tempe, Arizona 85287}
\email{kaliszewski@asu.edu}

\author[Quigg]{John Quigg}
\address{School of Mathematical and Statistical Sciences
\\Arizona State University
\\Tempe, Arizona 85287}
\email{quigg@asu.edu}

\date{\today}

\subjclass[2000]{Primary  46L55}
\keywords{
action, 
coaction, 
crossed product duality, 
ideal,
Morita equivalence}

\maketitle

\section{Introduction}\label{sec:intro}
Given a $\cst$-dynamical system $(A,G,\alpha)$ with $\I(A)$ the lattice of ideals of $A$, one can form a sub-lattice $\I_{\alpha}(A)$ of the action invariant ideals.  That is, $\alpha_s(I) \subset I$ for all $s \in G$.  However in the dual situation, there are many different ways to define coaction invariance.  Given a cosystem $(A,G,\delta)$, Nilsen in \cite{ndual} defines coaction invariance in terms of induced regular representations.  Namely if $I = \ker q$, then $I$ is Nilsen $\delta$-invariant provided $\ker q = \ker((q \otimes \lambda) \circ \delta)$.  Nilsen used this notion of coaction invariance to establish bijections between the $\alpha$ and $\delta$-invariant ideals of $A$ and the $\what \alpha$ and $\what \delta$ invariant ideals of the crossed products $A \rtimes_{\alpha} G$, $A \rtimes_{\delta} G$ respectively.  Recently in \cite{ladder} (and even more recently in \cite{qladder}) an analogous result was established using what is referred to as strong $\delta$-invariance in the literature.  Strong $\delta$-invariance of $I$ simply requires that $\delta|_{I}$ defines a coaction of $G$ on the $\cst$-algebra $I$.  These two notions coincide for amenable $G$ as we demonstrate, but easily fail to be whenever $G$ is non-amenable.  In 
\secref{sec:prelim}
 we establish the needed definitions and theorems.  
In \secref{comparing}, we show that Nilsen's definition of action invariance aligns with the standard notion and characterize when Nilsen and strong invariance are equal for coactions of groups.  We also establish a simple lattice isomorphism between the lattices of strong and Nilsen invariant ideals 
when $\delta$ is maximal or normal.  

\section{Preliminaries}\label{sec:prelim}
 Given a locally compact group $G$, we will always denote $\K \coloneqq \K(L^2(G))$.  The undecorated symbol $\otimes$ will always mean the minimal tensor product of $\cst$-algebras or the Hilbert space tensor product.  Moreover, $G$ will always denote a locally compact group, $A,B$ will always be $\cst$-algebras, $\alpha$ will always denote an action of $G$ on a $\cst$-algebra, $\delta$ will always denote a coaction of $G$ on a $\cst$-algebra, and actions or coactions decorated with a $\what{}$ symbol will always denote the dual coaction or dual action of $G$ on the corresponding crossed product $\cst$-algebra.  The symbol $M$ will either denote the multipliers of a $\cst$-algebra or the multiplication representation of $C_0(G)$.  It will be clear from context which is which. Throughout we use the notation
$\pi(A)B \coloneqq \{\pi(a)b \mid a \in A, b \in B\}$ for a $*$-homomorphism $\pi: A \to M(B)$. Finally, $\lambda: \cst(G) \to \cst_r(G)$ will always denote the integrated form of the left regular representation $\lambda: G \to \mathcal{B}(L^2(G))$ by $\lambda_s(\xi)(t) \coloneqq \xi(s^{-1}t)$.  The right regular representation $\rho: G \to \mathcal{B}(L^2(G))$ is given by $\rho_s\xi(t) = \Delta_G(s)^{\frac{1}{2}}\xi(ts)$ for $\Delta_G$ the modular function of $G$.    
 
 \begin{defn}
A coaction is an injective non-degenerate $*$-homomorphism 
$\delta: A \to M(A \otimes \cst(G))$ satisfying 
\begin{enumerate}[(a)]
    \item $\delta(A)(1_A \otimes \cst(G)) \subseteq A \otimes \cst(G)$;
    \item $(\delta \otimes \id_{\cst(G)}) \circ \delta = (\id_A \otimes \delta_G) \circ \delta$
\end{enumerate}
where $\delta_G: \cst(G) \to M(\cst(G) \otimes \cst(G))$ is the integrated form of the strictly continuous unitary homomorphism $s \mapsto s \otimes s$.  Finally, we say $\delta$ is continuous provided that \[\o\spn \{ \delta(A)(1_A \otimes \cst(G)) \} = A \otimes \cst(G).\] 
 \end{defn}
  We will always assume our coactions are continuous and many times may refer to the data $(A,G,\delta)$ as a coaction (similarly for actions).  Given an action $(A,G,\alpha)$ and a coaction $(A,G,\delta)$, we denote the full crossed product $\cst$-algebras as $A \rtimes_{\alpha} G$ and $A \rtimes_{\delta} G$ where full crossed products come with universal covariant homomorphisms $(i^{\alpha}_A,i^{\alpha}_G)$ for actions and $(j^{\delta}_A,j^{\delta}_G)$ for coactions respectively (we will sometimes write $j_A$ instead of $j_A^{\delta}$ if the context is clear, and similarly for the other maps mentioned; also, we frequently write $\rtimes G$ rather than $\rtimes_\delta G$, and similarly for the crossed product by the dual action.
  In fact, we can (and will) take $j_A = (\id_A \otimes \lambda) \circ \delta$ and $j_G = 1_A \otimes M$, see \cite[Definition~A.39]{enchilada}.  We denote the reduced crossed product of an action by $A \rtimes_{\alpha,r} G$ which is the image of $A \rtimes_{\alpha} G$ under the regular representation $\Lambda^{\alpha}_{A}: A \rtimes_{\alpha} G \to M(A \otimes \K)$, which is determined by 
\begin{enumerate}[(a)]
    \item $\Lambda^{\alpha}_A \circ i_A = (\id_A \otimes M^{-1}) \circ \oalpha$;
    \item $\Lambda^{\alpha}_A \circ i_G = 1_A \otimes \lambda$
  \end{enumerate}
where $\oalpha: A \to C_b(G,A) \subset M(A \otimes C_0(G)) \cong C_b(G,M^{\beta}(A))$ is the $C_0(G)$-coaction given by $\oalpha(a)(s) \coloneqq \alpha_s(a)$ and $M^{-1}(f)$ is multiplication by $f((\cdot)^{-1})$.  In other words, $\Lambda^{\alpha}_A = ((\id_A \otimes M^{-1}) \circ \oalpha) \rtimes (1_A \otimes \lambda)$. 
We will write $\Lambda_A$, or even $\Lambda$, if no confusion is likely.
There 
is
also 
the
\textit{canonical surjection} 
of the double crossed 
product
onto the stabilization of $A$ denoted by 
$\Phi_A^{\delta}: A \rtimes_{\delta} G \rtimes_{\what \delta} G \to A \otimes \K$ determined by 
\begin{enumerate}[(a)]
    \item $\Phi_A^{\delta} \circ i_{A \rtimes_{\delta} G} \circ j_A = (\id_A \otimes \lambda) \circ \delta$;
    \item $\Phi_A^{\delta} \circ i_{A \rtimes_{\delta} G} \circ j_G = 1_A \otimes M$;
    \item $\Phi_{A}^{\delta} \circ i_G = 1_A \otimes \rho$.
\end{enumerate}
 That is, $\Phi_A^{\delta} = [((\id_A \otimes \lambda) \circ \delta) \rtimes (1_A \otimes M)] \rtimes (1_A \otimes \rho)$.  See 
 \cite[Appendix~A]{enchilada} 
 and \cite{ndual} for details on crossed products and the canonical surjections.
 We write $\Phi$ or $\Phi_A$ when confusion is unlikely.
 \begin{defn}
Let $(A,G,\delta)$ be a coaction.
\begin{enumerate}[(a)]
    \item $\delta$ is called \textit{normal} if $j_A$ is faithful.
    \item $\delta$ is called \textit{maximal} if $\Phi_{A}$ is faithful.
\end{enumerate}
 \end{defn}

All coactions have canonical \textit{normalizations} and \textit{maximalizations}, denoted $(A^n,\delta^n)$ (which will appear in 
\thmref{thmid} below) and $(A^m,\delta^m)$, (see \cite{enchilada,fischer} for details). An important characterization of normality that we will use is the following.

\begin{thm}\label{thm3}
For any coaction $(A,G,\delta$),
there is an \iso\ $\Upsilon$
and a \sur\ $\Psi$
making the diagram
\[
\begin{tikzcd}
A\rtimes G\rtimes G 
\arrow[r,"\Phi"] \arrow[d,"\Lambda"']
&A\xt \K \arrow[d,"q\xt \id"] \arrow[dl,"\Psi"',dashed]
\\
A\rtimes G\rtimes_r G \arrow[r,"\Upsilon"',"\cong",dashed]
&A/\ker j_A\xt \K
\end{tikzcd}
\]
commute.
Moreover, $\delta$ is normal if and only if $\Psi$ is an isomorphism, equivalently if and only if $\ker\Phi=\ker\Lambda$.
\end{thm}
In the above, note that we always have $\ker\Phi\subset \ker\Lambda$.

To discuss Nilsen's notion of action and coaction invariance of (closed two-sided) ideals, we must introduce some notation and recall some concepts from her work.  Many of the following definitions and lemmas are directly from \cite[Section~1,2]{ndual}.

\begin{defn}
Let $\phi: A \to M(B)$ be a $*$-homomorphism.  Define
\begin{enumerate}[(a)]
    \item $\Res_{\phi}: \I(B) \to \I(A)$ by $\Res_{\phi}(\ker \pi) = \ker(\pi \circ \phi)$;
    \item $\Ext_{\phi}: \I(A) \to \I(B)$ by $\Ext_{\phi}(I) = \o\spn \{ B\phi(I)B\}$.
\end{enumerate}
\end{defn}

It is well-known that the definition of $\Res_{\phi}$ is independent of the choice of non-degenerate representation one chooses when writing an ideal of $B$ as a kernel.  For an action $(A,G,\alpha)$ or coaction $(A,G,\delta)$ and $\phi = j_A$ or $i_A$, we denote the $\Res$ map as $\Res_{\alpha} \coloneqq \Res_{i_A}$ and $\Res_{\delta} \coloneqq \Res_{j_A}$.  We use the same convention for $\Ext$.   

\begin{lem}
Let $(A,G,\alpha)$ and $(A,G,\delta)$ be an action and coaction respectively.  Moreover, let $\pi: A \to M(B)$ be a $*$-homomorphism.
\begin{enumerate}[\upshape (a)]
    \item The pair $((\pi \otimes M^{-1}) \circ \oalpha, 1_{B} \otimes \lambda)$ is covariant with integrated form denoted $\Ind_{\alpha}  \pi: A \rtimes_{\alpha} G \to M(B \otimes \K)$ called the induced regular homomorphism of $\pi$ associated to $\alpha$.
    \item The pair $((\pi \otimes \lambda) \circ \delta, 1_{B} \otimes M)$ is covariant with integrated form denoted $\Ind_{\delta}  \pi: A \rtimes_{\delta} G \to M(B \otimes \K)$ called the induced regular homomorphism of $\pi$ associated to $\delta$.
\end{enumerate}
When $B = \K(\H)$ for some Hilbert space $\H$, we say induced regular representation.
\end{lem}

The induced regular homomorphisms give rise to a map on the lattices of ideals $\Ind_{\alpha}: \I(A) \to \I(A \rtimes_{\alpha} G)$ by $\Ind_{\alpha}(\ker \pi) = \ker(\Ind_{\alpha} \pi)$.  $\Ind_{\delta}$ is defined in the analogous way.  We now are in a position to define and discuss the five different notions of action and coaction invariance of ideals:

\begin{defn} \label{definv}
Let $(A,G,\alpha)$ and $(A,G,\delta)$ be an action and coaction respectively.  Moreover, let $\ker q = I$ be a (closed two-sided) ideal of $A$ for $q$ the canonical quotient map, and set $\I(A)$ to be the lattice of such ideals of $A$.
\begin{enumerate}[(a)]
    \item $I$ is \textit{$\alpha$-invariant} provided that $\alpha_s(I) \subseteq I$ for all $s \in G$.  We denote the lattice of such ideals by $\I_{\alpha}(A)$;
    \item $I$ is \textit{Nilsen $\alpha$-invariant} provided that $I = \ker(\Ind_{\alpha}  q \circ i_A)$.  We denote the lattice of such ideals by $\I_{\alpha}^N(A)$;
    \item $I$ is \textit{weakly $\delta$-invariant} provided that $I \subset \ker((q \otimes \id_{\cst(G)}) \circ \delta)$.  We denote the lattice of such ideals by $\I_{\delta}^w(A)$;
    \item $I$ is \textit{strongly $\delta$-invariant} provided that $\delta|_{I}$ defines a coaction of $G$ on the $\cst$-algebra $I$.  We denote the lattice of such ideals by $\I_{\delta}^s(A)$;
    \item $I$ is \textit{Nilsen $\delta$-invariant} provided that $I = \ker(\Ind_{\delta}  q \circ j_A)$.  We denote the lattice of such ideals by $\I_{\delta}^N(A)$.
\end{enumerate}
\end{defn}
We omit the ``$(A)$'' and/or the subscripts $\alpha$ or $\delta$ when there is no ambiguity.

As Nilsen points out, her definition of action and coaction invariance is independent of the choice of representation.  Note that each of strong and Nilsen $\delta$-invariance imply weak $\delta$-invariance, see \cite[Section~3]{ndual} and \cite[Proposition~2.2]{ninv}.  Moreover, weak $\delta$-invariance is equivalent to there existing a coaction $\delta^I$ of $G$ on $A/I$ making the following diagram commute: \[\begin{tikzcd}
A \arrow[rr, "\delta"] \arrow[dd, "q"'] &  & M(A \otimes \cst(G)) \arrow[dd, "q \otimes \id_{\cst(G)}"] \\
                                        &  &                                                            \\
A/I \arrow[rr, "\delta^I"', dashed]     &  & M(A/I \otimes \cst(G)).                                    
\end{tikzcd}\]  This is the case since $I \subset \ker((q \otimes \id_{\cst(G)}) \circ \delta)$ actually implies equality, see \cite[Lemma~3.11]{exoticgroup}.  We must also mention a list of fundamental results from \cite{ndual} that we will use without much comment in the next section.

\begin{thm}[{\cite[Proposition~3.1,3.3, Corollary~3.2,3.4]{ndual}}] \label{thm2}
Let $(A,G,\alpha)$ and $(A,G,\delta)$ be an action and coaction respectively with $I \in \I(A)$.  
\begin{enumerate}[\upshape (a)]
    \item The following are equivalent:
    \begin{enumerate}[\upshape (i)]
        \item $\Res_{\delta} \circ \Ind_{\delta}(I) = I$;
        \item $\Res_{\delta} \circ \Ext_{\delta}(I) = I$;
        \item $I \in \I_{\delta}^N(A)$.
    \end{enumerate}
    \item $\Ext_{\delta} = \Ind_{\delta}$ on $\I_{\delta}^N(A)$.
    \item $\Res_{\delta}|_{\I_{\what{\delta}}(A \rtimes_{\delta} G)}$ is a lattice isomorphism onto $\I_{\delta}^N(A)$.
    \item $\Ind_{\delta}(\I_{\delta}^N(A)) = \I_{\what{\delta}}(A \rtimes_{\delta} G)$.
    \item Parts \(a\)--\(d\) are true for $(A,G,\alpha)$.
\end{enumerate}
\end{thm}

Finally, we mention the main result of \cite{ladder}, where we established lattice isomorphisms similar to Nilsen, but for strongly invariant ideals.

\begin{thm}{\cite[Theorem~3.2]{ladder}} \label{thmid}
Let $(A,G,\alpha)$ and $(A,G,\delta)$ be an action and coaction respectively.  
\begin{enumerate}[\upshape (a)]
    \item The assignments $\I_{\alpha}(A) \to \I_{\what{\alpha}}^s(A \rtimes_{\alpha} G)$  and $\I_{\alpha}(A) \to \I_{\what{\alpha}^n}^s(A \rtimes_{\alpha,r} G)$ given by $I \mapsto I \rtimes_{\alpha} G$ and $I \mapsto I \rtimes_{\alpha,r} G$ are lattice isomorphisms respectively.
    \item If $\delta$ is maximal or normal, then the assignment $\I_{\delta}^s(A) \to \I_{\what{\delta}}(A \rtimes_{\delta} G)$ given by $I \mapsto I \rtimes_{\delta} G$ is a lattice isomorphism.
\end{enumerate}
\end{thm}

Note that by $I \rtimes_{\alpha} G$, and $I \rtimes_{\delta} G$, we mean the isomorphic image of $I \rtimes_{\alpha|_{I}} G$ and $I \rtimes_{\delta|_{I}} G$ under the canonical inclusion $\iota: I \hookrightarrow A$ after one applies the full crossed product functor $\cdot \rtimes_{\alpha|_{I}} G$ and $\cdot \rtimes_{\delta|_{I}} G$ to $\iota$ respectively.  The same is true for $I \rtimes_{\alpha,r} G$ using the reduced crossed product functor.  In any case we will denote this by $\iota \rtimes G$ omitting the action or coaction.  These functors preserve equivariance of actions and coactions, see \cite[Remarks above Theorem~2.1]{ladder} and \cite[Lemma~A.16]{enchilada}.

\section{Comparing Strong and Nilsen Invariance}\label{comparing}
There are multiple relationships between the different notions of invariance given in definition \ref{definv}.  The story for actions is short and simple.  In particular, Nilsen's definition of invariance for an action $(A,G,\alpha)$ is equivalent to $\alpha$-invariance.

\begin{lem} \label{lem1}
Given an action $(A,G,\alpha)$, we have $\I_{\alpha}(A) = \Inn_{\alpha}(A)$.
\end{lem}

\begin{proof}
Let $\pi: A \to \mathcal{B}(\H)$ be a $*$-representation with $I = \ker \pi$.  Suppose first that $I \in \I_{\alpha}(A)$ so that $\alpha_s(I) \subset I$ for all $s \in G$.  Let $a \in I$.  Observe that $(\pi \otimes M^{-1}) \circ \alpha: A \to \mathcal{B}(\H \otimes L^2(G))  \cong \mathcal{B}(L^2(G,\H))$. Then for $s \in G$, and $\xi \in L^2(G,\H)$ \begin{align*}
(\pi \otimes M^{-1}) \circ \alpha(a)\xi(s) &= \pi(\alpha_{s^{-1}}(a))\xi(s) \\ & \in \pi(I)\xi(s) \\ &=\{0\}
\end{align*} so that $\pi \otimes M^{-1} \circ \alpha(a) = 0$ and hence $a \in \ker((\pi \otimes M^{-1}) \circ \alpha)$.  On the other hand if $a \in \ker((\pi \otimes M^{-1}) \circ \alpha)$, then for any $h \in \H$ with $\xi \in L^2(G,\H)$ and $\xi(e)=h$, we have \begin{align*}
0 &= (\pi \otimes M^{-1}) \circ \alpha(a)\xi(e) \\ &= \pi(\alpha_e(a))h \\ &= \pi(a)h.
\end{align*}
Hence $a \in \ker \pi$ so that $\ker \pi = \ker((\pi \otimes M^{-1}) \circ \alpha)$, implying that $I \in \Inn_{\alpha}(A)$.  Now suppose that $I \in \Inn_{\alpha}(A)$ and let $a \in I$ with $s \in G$.  It suffices to show $\pi(\alpha_{s^{-1}}(a))=0$.  Since $I = \ker((\pi \otimes M^{-1}) \circ \alpha)$ we have for any $h \in \H$ and fixed $s \in G$, one can pick $\xi \in L^2(G,\H)$ such that $\xi(s)=h$ so that \begin{align*}
0 &= (\pi \otimes M^{-1}) \circ \alpha(a)\xi(s) \\ &= \pi(\alpha_{s^{-1}}(a))h.
\end{align*}  This forces $\alpha_{s^{-1}}(a) \in \ker \pi = I$ for each $s \in G$.  Hence $I \in \I_{\alpha}(A)$, proving the equality.

\end{proof}

Thus for actions we will denote $\I_{\alpha}^N(A)$ as just $\I_{\alpha}(A)$. Characterizing when strong and Nilsen invariance are equivalent is a bit trickier.  To do so, we require a few lemmas. 

\begin{lem} \label{lemext}
    Let $(A,G,\delta)$ be a coaction and let $I \in \I^s$.  Then $\Ext_{\delta}(I) = I \rtimes_{\delta} G$
\end{lem}

\begin{proof}
Observe that since $I \in \I^s$, we have that $I \rtimes_{\delta|_{I}} G$ makes sense so that $I \rtimes_{\delta} G$ is a $\cst$-subalgebra of $A \rtimes_{\delta} G$.  In particular it's given by \begin{align*}
(\iota \rtimes G)(I \rtimes_{\delta|_{I}} G) &= (\iota \rtimes G)\left(\o\spn \left\{j^{\delta|_{I}}_I(I)j^{\delta|_{I}}_G(C_0(G))\right\} \right) \\ &= (j_A^{\delta} \circ \iota \rtimes j_G^{\delta})\left( \o\spn \left\{j^{\delta|_{I}}_I(I)j^{\delta|_{I}}_G(C_0(G))\right\} \right) \\ &= \o\spn \left\{j_A^{\delta}(I)j_G^{\delta}(C_0(G)) \right\}.
\end{align*}  In particular by $*$-invariance we have that \begin{align*}
I \rtimes_{\delta} G &= \o\spn \left\{j_A^{\delta}(I)j_G^{\delta}(C_0(G)) \right\} \\ &= \o\spn \left\{ j_G^{\delta}(C_0(G))j_A^{\delta}(I)\right\}.
\end{align*}
Thus, we compute to see that
\begin{align*}
\Ext_{\delta}(I) 
&= \o\spn \{ (A \rtimes_{\delta} G)j_A(I)(A \rtimes_{\delta} G) \}
\\ &= \o\spn \left\{j_A^{\delta}(A)j_G^{\delta}(C_0(G))j_A^{\delta}(I)j_A^{\delta}(A)j_G^{\delta}(C_0(G)) \right\} 
\\ &= \o\spn \left\{j_A^{\delta}(A)j_G^{\delta}(C_0(G))j_A^{\delta}(I)j_G^{\delta}(C_0(G)) \right\} 
\\ &= \o\spn \left\{j_A^{\delta}(A)j_G^{\delta}(C_0(G)) \o\spn \left\{j_A^{\delta}(I)j_G^{\delta}(C_0(G))\right\} \right\} 
\\ &= \o\spn \left\{j_A^{\delta}(A)j_G^{\delta}(C_0(G)) \o\spn \left\{j_G^{\delta}(C_0(G))j_A^{\delta}(I) \right\} \right\} 
\\ &= \o\spn \left\{j_A^{\delta}(A)j_G^{\delta}(C_0(G)) j_G^{\delta}(C_0(G))j_A^{\delta}(I) \right\} 
\\ &= \o\spn \left\{j_A^{\delta}(A)j_G^{\delta}(C_0(G))j_A^{\delta}(I) \right\} 
\\ &= \o\spn \left\{j_A^{\delta}(A) \o\spn \left\{j_G^{\delta}(C_0(G))j_A^{\delta}(I) \right\} \right\} \\ &= \o\spn \left\{j_A^{\delta}(A) \o\spn \left\{j_A^{\delta}(I)j_G^{\delta}(C_0(G)) \right\} \right\} \\ &= \o\spn \left\{j_A^{\delta}(A) j_A^{\delta}(I)j_G^{\delta}(C_0(G)) \right\} 
\\ &= \o\spn \left\{j_A^{\delta}(I)j_G^{\delta}(C_0(G)) \right\} 
\\ &= I \rtimes_{\delta} G.
\end{align*}
\end{proof}

\begin{lem}\label{strong then nielsen}
If $(A,G,\delta)$ is a coaction and 
$\I^s\subset \I^N$, then $\delta$ is normal and $\I^s=\I^N$.
\end{lem}

\begin{proof}
Since $\{0\} \in \I^s$, we conclude that it's also Nilsen invariant so that \begin{align*}
\{0\} 
&= \ker \id_A 
\\ &= \ker  \bigl( (\Ind_{\delta}  \id_A) \circ j_A \bigr)
\\ &= \ker  j_A.
\end{align*} 
Thus $\delta$ is normal.  Now, let $I \in \I^N$.  Then $\Ind_{\delta}(I) \in \I_{\what\delta}(A \rtimes_{\delta} G)$ so that $\Ind_{\delta}(I) = J \rtimes_{\delta} G$ for some $J \in \I^s \subset \I^N$ by \thmref{thmid}.  Hence $J \rtimes_{\delta} G = \Ext_{\delta}(J)$ by \lemref{lemext}.  However then \begin{align*}
I &= \Res_{\delta} \circ \Ind_{\delta}(I) \\ &= \Res_{\delta}(J \rtimes_{\delta} G) \\ &= \Res_{\delta} \circ \Ext_{\delta}(J) \\ &= J
\end{align*} by \thmref{thm2} so that $I=J \in \I^s$ and hence that $\I^s = \I^N$.   
\end{proof}

\begin{lem} \label{lem4}
Let $(A,G,\delta)$ be a coaction and $I \in \I^w$.  Then $I \in \I^N$ if and only if $\delta^I$ is normal.
\end{lem}

\begin{proof}
First suppose that $\ker q = I \in \I^N$ for $q: A \to A/I$ the canonical quotient map.  Then $I \in \I^w$ so that $\delta^I$ is a well-defined coaction of $G$ on $A/I$.  We want to prove $\jai  = (\id_{A/I} \otimes \lambda) \circ \delta^I$ is faithful.  Indeed, suppose that $\jai (x) = 0$.  Then $x = q(a)$ for some $a \in A$.  However then \begin{align*}
a &\in \ker \jai  \circ q \\ &= \ker((\id_{A/I} \otimes \lambda) \circ \delta^I \circ q) \\ &= \ker((\id_{A/I} \otimes \lambda) \circ (q \otimes \id_{\cst(G)}) \circ \delta) \\ &= \ker((q \otimes \lambda) \circ \delta) \\ &= \ker q \\ &= I 
\end{align*} so that $a \in I$, where we used Nilsen invariance only in the fourth equality.  However then $x = q(a)=0$ proving $\jai $ is faithful.  Hence, $\delta^I$ is normal.  

Conversely, suppose $\delta^I$ is normal.  Since $\ker  \jai  \circ q = \ker((q \otimes \lambda) \circ \delta)$, it's obvious that $\ker  q \subset \ker((q \otimes \lambda) \circ \delta)$.  Thus, suppose that $a \in \ker((q \otimes \lambda) \circ \delta) = \ker(\jai  \circ q)$.  Then $\jai (q(a)) = 0$ which by faithfulness of $\jai $ forces $q(a)=0$, i.e that $a \in I = \ker  q$.  Thus, \begin{align*}
I &= \ker  q \\ &= \ker((q \otimes \lambda) \circ \delta)
\end{align*} implying that $I \in \I^N$.  
\end{proof}

This characterization of Nilsen invariance is highly important.  It's a well-known fact that for a maximal coaction $(A,G,\delta)$ and $I \in \I^s$ that $\delta^I$ is also maximal.  If $(A,G,\delta)$ and $I \in \I^s$, knowing when $I$ is Nilsen invariant is then equivalent to knowing when $\delta$ passes to a normal coaction $\delta^I$ on the quotient.  To this end, we characterize exactly when these two notions of invariance are equal.      

\begin{thm} \label{thm1}
For a coaction $(A,G,\delta)$ we have 
$\I^s=\I^N$
if and only if $\delta$ is normal and the sequence 
\[
\begin{tikzcd}[column sep=scriptsize]
0 \arrow[r] 
& {I \rtimes G \rtimes_r G} 
\arrow[r, "\iota \rtimes G \rtimes_r G"] 
& {\agrg} \arrow[r, "\qgrg"] 
& {A/I \rtimes G \rtimes_r G} \arrow[r] 
& 0
\end{tikzcd}
\]
is short exact for all $I \in \I^s$, where $q: A \to A/I$ is the canonical quotient and $\iota: I \hookrightarrow A$ is the canonical inclusion. 
\end{thm}

\begin{proof}
First suppose that $\delta$ is normal and the exactness condition holds.
To show that $\I^s=\I^N$, by \lemref{strong then nielsen} it suffices to
let $I \in \I^s(A)$ and deduce that $I\in \I^N$. By \lemref{lem4} it suffices to show that the induced coaction $\delta^I$ on $A/I$ is normal.  By \thmref{thm3} this is equivalent to 
showing that $\ker \laig\subset \ker\pai$.

Since $\delta$ is normal, we have that $\delta_I$ is too since $\ia = j_A \circ \iota$.  In particular, we have isomorphisms
\begin{align*}
I\xt \K
&\overset{\Psi_I}\cong
\igrg
\\
A\xt \K
&\overset{\Psi_A}\cong
\agrg.
\end{align*}
Note by \cite{ladder} we know that 
$\igrg$\footnote{and recall that this means the natural image in $A\rtimes G\rtimes_r G$ of the reduced crossed product by $\what{\delta_I}$}
is an ideal of $\agrg$ and thus we obtain an isomorphism 
\[
\phi:  \agrg
\bigg/ \igrg
\to (A \otimes \K) \big/ (I \otimes \K)
\]
by $\phi(x + [I]) = \Psi_A\inv(x) + (I \otimes \K)$ where $[I]$ denotes the coset $\igrg$.  Since $\K$ is an exact $\cst$-algebra, we also get an isomorphism 
\[
\psi: (A \otimes \K) \big/ (I \otimes \K) \to A/I \otimes \K
\]
by $\psi(b + (I \otimes \K)) = (q \otimes \id_{\K})(b)$.  Thus, we obtain an isomorphism
\[
\agrg \bigg/ \igrg
\xrightarrow{\psi \circ \phi}
A/I \otimes \K.
\]
By the exactness condition, the sequence
\[
\begin{tikzcd}[column sep=scriptsize]
0 \arrow[r] 
& \igrg \arrow[r, "\iota \rtimes G \rtimes_r G"] 
& \agrg \arrow[r, "\qgrg"]
& \aigrg \arrow[r] 
& 0
\end{tikzcd}
\]
is short exact.  In particular, we obtain the obvious isomorphism
\[
\eta: \agrg
\bigg/ 
\igrg
\to 
\aigrg
\]
given by $\eta(x+[I]) = (q \rtimes G \rtimes_r G)(x)$.  
It suffices to show that the diagram
\begin{equation}\label{kernel}
\begin{tikzcd}
\aigrg \arrow[r,"\eta\inv","\cong"']
&\agrg\bigg/\igrg \arrow[d,"\psi\circ\phi","\cong"']
\\
\aigg \arrow[r,"\pai"'] \arrow[u,"\laig"]
&A/I\xt\K
\end{tikzcd}
\end{equation} 
commutes.
First
observe that since 
\[
\begin{tikzcd}
0 \arrow[r] 
& I \arrow[r, "\iota"] 
& A \arrow[r, "q"] 
& A/I \arrow[r] 
& 0
\end{tikzcd}
\]
is short exact and $I$ is strongly invariant, we have by \cite[Theorem~2.3]{ninv} that
\[
\begin{tikzcd}
0 \arrow[r] 
& I \rtimes G \arrow[r, "\iota \rtimes G"] 
& A \rtimes G \arrow[r, "q \rtimes G"] 
& A/I \rtimes G \arrow[r] 
& 0
\end{tikzcd}
\]
is short exact. Then note that $I\rtimes G$ is invariant under the dual action by \cite[Proposition 2.3(b)]{ladder}. Applying \cite[Proposition~3.19]{dana}, we have that 
\[
\begin{tikzcd}[column sep=scriptsize]
0 \arrow[r] 
& \igg \arrow[r, "\iota \rtimes G \rtimes G"] 
& \agg \arrow[r, "\qgg"] 
& \aigg \arrow[r] 
& 0
\end{tikzcd}
\] 
is short exact, and thus that 
\[
\aigg 
= (\qgg)\l(\agg\r).
\] 
Now, it's a basic fact that for action equivariant $*$-homomorphisms $f: A \to B$ with crossed product $f \rtimes G: A \rtimes G \to B \rtimes G$, that the following square commutes \[\begin{tikzcd}
A \rtimes G \arrow[dd, "\Lambda_A"'] \arrow[rr, "f \rtimes G"] &  & B \rtimes G \arrow[dd, "\Lambda_B"] \\
                                                               &  &     \\
A \rtimes_r G \arrow[rr, "f \rtimes_r G"']                     &  & B \rtimes_r G.                      
\end{tikzcd}\]  In fact, this is more or less what the proof of \cite[Lemma~A.16]{enchilada} states.  Thus, 
we have 
\begin{align*}
\iaig\circ \jai\circ q
&=(\qgg)\circ \iag\circ \ja
\\
\iaig\circ \jg
&=(\qgg)\circ \iag\circ \jg
\\
i_G^{\what{\delta^I}}
&=(\qgg)\circ i_G^{\what\delta}
\end{align*}
so we apply the regular representations,
remembering the rule
\[
\laig\circ (\qgg)=(\qgrg)\circ\lag,
\]
to get
\begin{align*}
\laig \circ \iaig \circ \jai\circ q
&=(\qgrg) \circ \lag \circ \iag\circ \ja
\\
\laig \circ \iaig\circ \jg
&=(\qgrg) \circ \lag \circ \iag\circ \jg
\\
\laig
\circ i_G^{\what{\delta^I}}
&=\laig
\circ (\qgg)\circ i_G^{\what\delta}
\\&=(\qgrg)
\circ \lag
\circ i_G^{\what\delta},
\end{align*}
so
for $a\in A$, $f\in C_0(G)$, and $y\in C^*(G)$,
applying $\psi\circ\phi\circ \eta^{-1}$ we get 
\begin{align*}
&(\psi\circ\phi\circ\eta^{-1})\circ \laig \circ \iaig\circ \jai \circ q(a)
\\&\quad=(\psi\circ\phi)\biggl(\lag\circ \iag\circ\ja(a)+[I]\biggr)
\\&\quad=\psi\biggl(\Psi_A\inv\circ \lag\circ \iag\circ\ja(a)+(I\xt\K)\biggr)
\\&\quad=\psi\biggl(\pa\circ \iag\circ\ja(a)+(I\xt\K)\biggr)
\\&\quad=(q\xt \id)\circ \pa\circ \iag\circ \ja(a)
\\&\quad=\pai\circ\iaig\circ \jai \circ q(a),
\end{align*}
\begin{align*}
&(\psi\circ\phi\circ\eta^{-1})\circ \laig \circ \iaig\circ \jg(f)
\\&\quad=(\psi\circ\phi)\biggl(\lag\circ \iag\circ\jg(f)+[I]\biggr)
\\&\quad=\psi\biggl(\Psi_A\inv\circ \lag\circ \iag\circ\jg(f)+(I\xt\K)\biggr)
\\&\quad=\psi\biggl(\pa\circ \iag\circ\jg(f)+(I\xt\K)\biggr)
\\&\quad=(q\xt \id)\circ \pa\circ \iag\circ \jg(f)
\\&\quad=\pai\circ\iaig\circ \jg(f),
\end{align*}
and
\begin{align*}
&(\psi\circ\phi\circ\eta^{-1})\circ \laig \circ i_G^{\what{\delta^I}}(y)
\\&\quad=(\psi\circ\phi)\biggl(\lag\circ i_G^{\what\delta}(y)+[I]\biggr)
\\&\quad=\psi\biggl(\Psi_A\inv\circ \lag\circ i_G^{\what\delta}(y)+(I\xt\K)\biggr)
\\&\quad=\psi\biggl(\pa\circ i_G^{\what\delta}(y)+(I\xt\K)\biggr)
\\&\quad=(q\xt \id)\circ \pa\circ i_G^{\what\delta}(y)
\\&\quad=\pai\circ i_G^{\what{\delta^I}}(y),
\end{align*}
thus verifying commutativity of 
\diagref{kernel}.

Conversely, suppose that $\I^s = \I^N$.  Then $\delta$ is normal by \lemref{lem4}. Now, let $I \in \I^s$.  We must show that
\[
\begin{tikzcd}[column sep=scriptsize]
0 \arrow[r] 
& \igrg
\arrow[r, "\iota \rtimes G \rtimes_r G"] 
& \agrg
\arrow[r, "\qgrg"] 
& \aigrg
\arrow[r] 
& 0
\end{tikzcd}
\]
is exact.
Since this sequence is a quotient of the short exact sequence
\[
\begin{tikzcd}[column sep=scriptsize]
0 \arrow[r] 
& \igg
\arrow[r, "\iota \rtimes G \rtimes G"] 
& \agg
\arrow[r, "\qgg"] 
& \aigg
\arrow[r] 
& 0,
\end{tikzcd}
\]
by routine diagram chasing we see that it is enough to show that
\begin{equation}\label{enough}
\ker \subset \igrg,
\end{equation}
where $\ker$ denotes the kernel of $\qgrg$.
Note that
equivariance of the regular representations implies that
$\qgrg$ maps $\agrg$ onto $\aigrg$.
In particular, we have the isomorphism
\[
\tau: \agrg
\bigg/ \ker 
\to \aigrg
\]
given
by $\tau(x + \ker) = \qgrg(x)$.
On the other hand since $\delta$ is normal, we 
have the isomorphisms
\begin{align*}
I\otimes \K
&\overset{\Psi_I}{\cong}
\igrg
\\
A \otimes \K
&\overset{\Psi_A}{\cong}
\agrg,
\end{align*}
and thus we get the isomorphism
\[
\agrg
\bigg/ 
\igrg
\xrightarrow{\psi \circ \phi}
A/I \otimes \K.
\]
However, since $I$ is also Nilsen invariant, we have that $\delta^I$ is normal, and thus we can apply \thmref{thm3} to 
it:
\[
A/I \otimes \K
\overset{\Psi_{A/I}}{\cong} 
\aigrg.
\]
Putting this together, we have a diagram
\[\begin{tikzcd}
                                                                                                     &  & A/I \otimes \K \arrow[lldd, "\Psi_{A/I}"',"\cong"]
\\
                                                                                                     &  &                                                                                                                                                                          \\
\aigrg
\arrow[rrdd, "\cong","\tau\inv"']
&  & 
\agrg \bigg/ \igrg
\arrow[dd, "\cong"',"\tilde{\eta}", dashed] 
\arrow[uu, "\cong","\psi \circ \phi"'] \\
                                                                                                     &  &                                                                                                                                                                          \\
                                                                                                     &  & 
\agrg \bigg/ \ker,
\end{tikzcd}\]
which we use to define the \iso\ $\tilde{\eta}$.
We claim that in fact $\tilde{\eta}$ is given by $\tilde{\eta}(x+[I]) = x + \ker$.
We observe that
\begin{align*}
&\tau\inv \circ \Psi_{A/I} \circ \psi \circ \phi(x + [I]) 
\\ & \quad = \tau\inv \circ \Psi_{A/I} \circ \psi\l(\Psi_A\inv(x) + (I \otimes \K)\r) 
\\ & \quad = \tau\inv \circ \Psi_{A/I} \circ (q \otimes \id_{\K}) \circ \Psi_A\inv(x). 
\end{align*}
Now observe that for $x \in A \rtimes_{\delta} G \rtimes_{\what \delta,r} G$ with $x = \lag(y)$, we have
\begin{align*}
&(q \otimes \id_{\K}) \circ \Psi_A\inv(x) 
\\ & \quad = (q \otimes \id_{\K}) \circ \Psi_A\inv \circ \lag(y) 
\\ & \quad = (q \otimes \id_{\K}) \circ \pa(y)
\\ & \quad = \pai \circ (\qgg)(y)
\end{align*}
where the last equality follows from
earlier computations.
On the other hand, we see that
\begin{align*}
&\Psi_{A/I}\inv \circ (\qgrg)(x) 
\\ & \quad = \Psi_{A/I}\inv \circ (\qgrg) \circ \lag(y) 
\\ & \quad = \Psi_{A/I}\inv \circ \laig \circ (\qgg)(y) 
\\ & \quad = \pai \circ (\qgg)(y).
\end{align*}
Thus, we conclude that
\[
(q \otimes \id_{\K}) \circ \Psi_A\inv
= \Psi_{A/I}\inv \circ (\qgrg)
\]
so computing we have
\begin{align*}
&\tau\inv \circ \Psi_{A/I} \circ (q \otimes \id_{\K}) \circ \Psi_A\inv(x) 
\\ & \quad = \tau\inv \circ \Psi_{A/I} \circ \Psi_{A/I}\inv \circ (\qgrg)(x) 
\\ & \quad = \tau\inv \circ (\qgrg)(x) 
\\ & \quad = x+\ker.
\end{align*}
Hence, the isomorphism $\tilde{\eta}$ defined in the above diagram is given by $\tilde{\eta}(x+[I]) = x + \ker$, proving the claim. Thus, if $x \in \ker$, then $\tilde{\eta}(x+[I]) = x+\ker = 0$ so that $x+[I]=0$ and hence $x \in \igrg$ by injectivity of $\tilde{\eta}$,
and we have finally established the inclusion \eqref{enough}. 
\end{proof}

In particular, if one is willing to work in the (very broad) class of $\cst$-exact groups in the sense of Kirchberg \cite{exact}, then normal coactions always pass to normal coactions on the quotients of strongly invariant ideals.

\begin{cor}
If $(A,G,\delta)$ is a normal coaction and $G$ is $\cst$-exact in the sense of Kirchberg, then $\I^s = \I^N$.  In particular, for every strongly $\delta$-invariant ideal $I$ of $A$, the induced coaction $\delta^I$ on the quotient is normal.
\end{cor}

\begin{proof}
Exactness of $G$ implies the exactness property in \thmref{thm1}.
\end{proof}

Finally, we mention that one can construct an explicit bijection of ideals between the strongly invariant ideals and Nilsen invariant ideals under the assumption of normality or maximality of the coaction.  Indeed, let $(A,G,\delta)$ be a maximal or normal coaction.  By \thmref{thmid} we have a lattice isomorphism 
$\Lad: \I^s \to \I_{\what{\delta}}(A\rtimes G)$
and by \thmref{thm2} we have a bijection 
$\Res_{\delta}: \I_{\what{\delta}}(A \rtimes G) \to \I^N$
using \lemref{lem1}.
Then the diagram
\[
\begin{tikzcd}
\I^s \arrow[rr, "\Lad"] \arrow[rrdd, dashed] &  & 
\I_{\what{\delta}}(A \rtimes G) \arrow[dd, "\Res_{\delta}"] 
\\
                                                         &  &                                                                      
\\
                                                         &  & \I^N                                                    
\end{tikzcd}
\]
defines  the desired bijection.
Notice however that if $\ker q = I \in \I^s$, then
\[
\begin{tikzcd}
0 \arrow[r] 
& I \rtimes G \arrow[r, "\iota \rtimes G"] 
& A \rtimes G \arrow[r, "q \rtimes G"] 
& A/I \rtimes G \arrow[r] 
& 0
\end{tikzcd}
\]
is exact as seen in the above proof.  In particular, $I \rtimes G = \ker(q \rtimes G)$ and so the above bijection is given by $I \mapsto \Res_{\delta} \circ \Ind_{\delta}(I)$.  Indeed
\begin{align*}
I 
&= \ker  q \mapsto \Res_{\delta}(I \rtimes G) 
\\ &= \Res_{\delta}(\ker(q \rtimes G)) 
\\ &= \ker((q \rtimes G) \circ j_A) 
\\ &= \ker((\jai  \circ q \rtimes j_G) \circ j_A) 
\\ &= \ker(\jai  \circ q) 
\\ &= \ker((q \otimes \lambda) \circ \delta) 
\\ &= \ker(\Ind_{\delta}  q \circ j_A)
\\ &= \Res_{\delta} \circ \Ind_{\delta}(\ker  q)
\\ &= \Res_{\delta} \circ \Ind_{\delta}(I)
\end{align*}
where the second to last equality follows from the proof of \lemref{lem1}.  We also see that this map is the identity whenever $\I^s = \I^N$ by \thmref{thm2}. Thus, we have proven the following theorem.

\begin{thm}
Let $(A,G,\delta)$ be a maximal or normal coaction.  Then the map $\I^s \to \I^N$ defined by $I \mapsto \Res_{\delta} \circ \Ind_{\delta}(I)$ is a bijection, which is equal to the identity whenever $\I^s = \I^N$.  In particular, this map is the identity whenever $\delta$ is normal and the exactness condition from \thmref{thm1} holds.  
\end{thm}


\providecommand{\bysame}{\leavevmode\hbox to3em{\hrulefill}\thinspace}
\providecommand{\MR}{\relax\ifhmode\unskip\space\fi MR }
\providecommand{\MRhref}[2]{%
  \href{http://www.ams.org/mathscinet-getitem?mr=#1}{#2}
}
\providecommand{\href}[2]{#2}


\end{document}